\documentclass{amsart}%
\usepackage{amssymb}
\usepackage{amsfonts}
\usepackage{amsmath}
\usepackage{graphicx}%
\providecommand{\U}[1]{\protect\rule{.1in}{.1in}}
\newtheorem{theorem}{Theorem}
\theoremstyle{plain}

\newtheorem{definition}{Definition}

\newtheorem{lemma}{Lemma}

\newtheorem{problem}{Problem}

\newtheorem{remark}{Remark}

\numberwithin{equation}{section}
\begin{document}
\title[ ]{The Hardy-Littlewood-P\'{o}lya inequality of majorization in the context of
$\mathbf{\omega}$-$\mathbf{m}$-star-convex functions}
\author{Geanina Maria L\u achescu}
\address{University of Craiova, Department of Mathematics, A.I. Cuza Street 13, Craiova
200585, ROMANIA}
\email{lachescu.geanina@yahoo.com}
\author{Ionel Roven\c ta}
\address{University of Craiova, Department of Mathematics, A.I. Cuza Street 13, Craiova
200585, ROMANIA}
\email{ionelroventa@yahoo.com}
\thanks{}
\date{January 24, 2021}
\dedicatory{ }
\begin{abstract}
The Hardy-Littlewood-P\'{o}lya inequality of majorization is extended for the
$\mathbf{\omega}$-$\mathbf{m}$-star-convex functions to the framework of
ordered Banach spaces. Several open problems which seem of interest for
further extensions of the Hardy-Littlewood-P\'{o}lya inequality are also included.

\end{abstract}
\maketitle

\section{Introduction}

The Hardy-Littlewood-P\'{o}lya theorem of majorization is an important result
in convex analysis that lies at the core of majorization theory, a subject
that attracted a great deal of attention due to its numerous applications in
mathematics, statistics, economics, quantum information etc. See
\cite{MO1979}, \cite{MOA2011}, \cite{NP2018}, \cite{NB}, \cite{PPT}, \cite{Pe}
and \cite{Ru} to cite just a few books treating this topic.

The relation of majorization was initially formulated as a relation between
the pairs of vectors with real entries rearranged downward, but nowadays
prevails its formulation as a preordering of probability measures.

For the reader's convenience we briefly recall here the most basic facts
concerning the theory of majorization.

Given two discrete probability measures $\mu=\sum_{k=1}^{N}\lambda_{k}%
\delta_{\mathbf{x}_{k}}$ and $\nu=\sum_{k=1}^{N}\lambda_{k}\delta
_{\mathbf{y}_{k}},$ supported by a compact interval $[a,b],$ we say that $\mu$
is \emph{majorized} by $\mathbf{\nu}$ $($denoted $\mu\prec\nu)$ if the
following three conditions are fulfilled:%
\[
\quad%
\begin{array}
[c]{ll}%
(M1) & \mathbf{x}_{1}\geq\mathbf{x}_{2}\geq\cdots\geq\mathbf{x}_{N}\\
(M2) & \sum_{i=1}^{k}\lambda_{i}\mathbf{x}_{i}\leq\sum_{i=1}^{k}\lambda
_{i}\mathbf{y}_{i}\quad\text{for }k=1,\dots,N;\text{ and}\\
(M3) & \sum_{i=1}^{N}\lambda_{i}\mathbf{x}_{i}=\sum_{i=1}^{N}\lambda
_{i}\mathbf{y}_{i}.
\end{array}
\]
When only conditions $(M1)$ and $(M2)$ occur, we say that $\mu$ is
\emph{weakly} \emph{majorized} by $\mathbf{\nu}$ $($denoted $\mu\prec_{w}%
\nu).$

Hardy, Littlewood and P\'{o}lya \cite{HLP} used a stronger formulation of
$(M1),$ by asking also that $\mathbf{y}_{1}\geq\mathbf{y}_{2}\geq\cdots
\geq\mathbf{y}_{N}$. Later, their result was improved by Maligranda,
Pe\v{c}ari\'{c} and Persson \cite{MPP1995} who were able to prove that
\begin{equation}
\mu\prec\nu\text{ implies}\int_{a}^{b}f~\mathrm{d}\mu=\sum_{k=1}%
^{N}f(\mathbf{x}_{k})\leq\int_{a}^{b}f~\mathrm{d}\nu=\sum_{k=1}^{N}%
f(\mathbf{y}_{k}), \tag{HLP}\label{HLP}%
\end{equation}
for all continuous convex functions $f:[a,b]\rightarrow\mathbb{R}.$ Moreover,
the same conclusion holds in the case of weak majorization and the convex and
nondecreasing functions.

Nowadays the inequality \ref{HLP} is known as the Hardy-Littlewood-P\'{o}lya
inequality of majorization.

In the early 1950s, the Hardy-Littlewood-P\'{o}lya inequality was extended by
Sherman \cite{She1951} to the case of continuous convex functions of a vector
variable by using a much larger concept of majorization, based on matrices
stochastic on lines. The full details can be found in \cite{NP2018}, Theorem
4.7.3, p. 219. Over the years, many other generalizations in the same vein
were published. See, for example, \cite{Bra}, \cite{NP2006}, \cite{NR2014},
\cite{NR2015a}, \cite{NR2015b}, \cite{NR2017} and \cite{Nie}.

As was noticed in \cite{N2021} and \cite{NO}, the Hardy-Littlewood-P\'{o}lya
inequality of majorization can be extended to the framework of convex
functions defined on ordered Banach spaces alongside the conditions
$(M1)-(M3)$. The aim of the present paper is to prove that the same works for
the larger class of $\omega$-$m$-star-convex functions.

The main features of these functions make the objective of Section 2.

In Section 3 we present different types of majorization relations in ordered
Banach spaces. The corresponding extensions of the Hardy-Littlewood-P\'{o}lya
inequality constitute the objective of Section 4. The paper ends by mentioning
several open problems which seem of interest for further extensions of the
Hardy-Littlewood-P\'{o}lya inequality.

\section{Preliminaries on $\omega$-$m$-star-convex functions}

Throughout this paper $E$ is a Banach space and $C$ is a convex subset of it.

\begin{definition}
\label{defmstar}Let $m$ be a real parameter belonging to the interval $(0,1].$
A function $\Phi:C\rightarrow\mathbb{R}$ is said to be a perturbed
$m$-star-convex function with modulus $\omega:\mathbb{[}0,\infty
\mathbb{)}\rightarrow\mathbb{R}$ $($abbreviated, $\omega$-$m$-star-convex
function$)$ if it verifies an estimate of the form%
\[
\Phi((1-\lambda)\mathbf{x}+\lambda m\mathbf{y})\leq(1-\lambda)\Phi
(\mathbf{x})+m\lambda\Phi(\mathbf{y})-m\lambda(1-\lambda)\omega\left(
\left\Vert \mathbf{x}-\mathbf{y}\right\Vert \right)  ,\text{\quad}%
\]
for all $\mathbf{x},\mathbf{y}\in C$ and $\lambda\in(0,1).$

The $\omega$-$m$-star-convex functions associated to an identically zero
modulus will be called $m$-star-convex. They verify the inequality%
\[
\Phi((1-\lambda)\mathbf{x}+\lambda m\mathbf{y})\leq(1-\lambda)\Phi
(\mathbf{x})+m\lambda\Phi(\mathbf{y}),
\]
for all $\mathbf{x},\mathbf{y}\in C$ and $\lambda$ $\in$ $(0,1)$.
\end{definition}

Notice that the usual convex functions represent the particular case of
$m$-star-convex functions where $m=1.$ On the other hand every convex function
is $m$-star-convex (for every $m\in(0,1])$ if $\mathbf{0}\in C$ and
$\Phi(\mathbf{0})\leq0.$ Indeed, we have
\begin{align*}
\Phi((1-\lambda)\mathbf{x}+\lambda m\mathbf{y})  &  =\Phi((1-\lambda
)\mathbf{x}+\lambda m\mathbf{y+(\lambda-\lambda}m\mathbf{)0})\\
&  \leq(1-\lambda)\Phi(\mathbf{x})+m\lambda\Phi(\mathbf{y})+\mathbf{(\lambda
-\lambda}m\mathbf{)\Phi(0})\\
&  =(1-\lambda)\Phi(\mathbf{x})+m\lambda\Phi(\mathbf{y}).
\end{align*}

Every $\omega$-$m$-star-convex function associated to a modulus\emph{ }%
$\omega\geq0$ is necessarily $m$-star-convex. The $\omega$-$m$-star-convex
functions whose moduli $\omega$ are strictly positive except at the origin
(where $\omega(0)=0)$ are usually called \emph{uniformly }$m$%
\emph{-star-convex}.\emph{ }In their case the definitory inequality is strict
whenever $\mathbf{x}\neq\mathbf{y}$ and $\lambda\in(0,1).$

By reversing the inequalities, one obtain the notions of $\omega$%
-$m$-\emph{star}-\emph{concave function }and\emph{ }of \emph{uniformly }%
$m$-\emph{star}-\emph{concave function.}

The theory of $m$-star-convex functions was initiated by Toader \cite{To1985},
who considered only the case of functions defined on real intervals. For
additional results in the same setting see \cite{MST} and the references therein.

\

A simple example of $(16/17)$-star-convex function which is not convex is
\begin{equation}
f:[0,\infty)\rightarrow\mathbb{R},\text{\quad}f(x)=x^{4}-5x^{3}+9x^{2}%
-5x.\label{ExMocanu}%
\end{equation}
See \cite{MST}, Example 2. Note that if $\Phi:C\rightarrow\mathbb{R}$ and
$\Psi:C\rightarrow\mathbb{R}$ are $\omega$-$m$-star-convex functions and
$\alpha,\beta\in\mathbb{R}_{+},$ then%
\[
\alpha\Phi+\beta\Psi\,\text{and }\sup\left\{  \Phi,\Psi\right\}
\]
are functions of the same nature. So is%
\[
\Phi\times\Psi:C\times C\rightarrow\mathbb{R},\text{\quad}\left(  \Phi
\times\Psi\right)  \left(  \mathbf{x},\mathbf{y}\right)  =\Phi(\mathbf{x}%
)+\Psi(\mathbf{y}).
\]
The class of $\omega$-$m$-star-convex functions is also stable under pointwise
convergence (when this exists).

Assuming $C\subset E$ is a convex cone with vertex at the origin, the
\emph{perspective }of a function%
\index{function!perspective}
$f:C\rightarrow\mathbb{R}$\ is the positively homogeneous function%
\[
\widetilde{f}:C\times(0,\infty)\rightarrow\mathbb{R},\text{\quad}\widetilde
{f}(\mathbf{x},t)=tf\left(  \frac{\mathbf{x}}{t}\right)  .
\]

\begin{lemma}
\label{lempersp}The perspective of every $m$-star-convex/concave function is a
function of the same nature.
\end{lemma}

\begin{proof}
Indeed, assuming (to make a choice) that $f$ is $\omega$-$m$-star-convex, then
for all $(\mathbf{x},s),(\mathbf{y},t)\in C\times(0,\infty)$ and $\lambda
\in\lbrack0,1]$ we have%
\begin{align*}
f\left(  \frac{(1-\lambda)\mathbf{x}+\lambda m\mathbf{y}}{(1-\lambda)s+\lambda
mt}\right)   &  =f\left(  \frac{(1-\lambda)s}{(1-\lambda)s+\lambda mt}%
\cdot\frac{\mathbf{x}}{s}+\frac{\lambda mt}{(1-\lambda)s+\lambda mt}\cdot
\frac{\mathbf{y}}{t}\right) \\
&  \leq\frac{(1-\lambda)s}{(1-\lambda)s+\lambda mt}f\left(  \frac{\mathbf{x}%
}{s}\right)  +\frac{\lambda mt}{(1-\lambda)s+\lambda mt}f\left(  \frac{y}%
{t}\right)
\end{align*}
that is,
\[
\widetilde{f}((1-\lambda)\mathbf{x}+\lambda m\mathbf{y},(1-\lambda)s+\lambda
mt)\leq(1-\lambda)\widetilde{f}(\mathbf{x},s)+\lambda m\widetilde
{f}(\mathbf{y},t).
\]

\end{proof}

Lemma \ref{lempersp}, allows us easily to produce nontrivial examples of
$m$-star-convex functions of several variables with some nice properties. For
example, starting (\ref{ExMocanu}), we conclude that%
\[
\Phi(x,t)=\frac{x^{4}-5x^{3}t+9x^{2}t^{2}-5xt^{3}}{t^{3}}%
\]
is a $(16/17)$-star-convex function on $[0,\infty)\times(0,\infty).$

Under the presence of G\^{a}teaux differentiability, the $\omega$%
-$m$-star-convex functions generate specific gradient inequalities that play a
prominent role in our generalization of the Hardy-Littlewood-P\'{o}lya
inequality of majorization.

\begin{lemma}
\label{lem1} Suppose also that $C$ is an open convex subset of the Banach
space $E$ and $\Phi:C\rightarrow\mathbb{R}$ is a function both G\^{a}teaux
differentiable and $\omega$-$m$-star-convex. Then \emph{ }
\begin{equation}
m\Phi(\mathbf{y})\geq\Phi(\mathbf{x})+d\Phi(\mathbf{x})(m\mathbf{y}%
-\mathbf{x})+m\omega\left(  \left\Vert \mathbf{x}-\mathbf{y}\right\Vert
\right)  ,\label{grad_ineq}%
\end{equation}
for all points $\mathbf{x,y}\in C.$
\end{lemma}

\begin{proof}
Indeed, we have
\[
\frac{\Phi((1-\lambda)\mathbf{x}+m\lambda\mathbf{y})-\Phi(\mathbf{x})}%
{\lambda}\leq-\Phi(\mathbf{x})+m\Phi(\mathbf{y}) -m(1-\lambda)\omega\left(
\left\Vert \mathbf{x}-\mathbf{y}\right\Vert \right)
\]
and the proof ends by passing to the limit as $\lambda\rightarrow0+.$
\end{proof}

\begin{remark}
\label{rem1}Lemma \ref{lem1} shows that the critical points $\mathbf{x}$ of
the differentiable $\omega$-$m$-star--convex functions for which $\omega\geq0$
verify the property%
\[
m\inf_{\mathbf{y}\in C}\Phi(\mathbf{y})\geq\Phi(\mathbf{x}).
\]

\end{remark}

Unlike the case of convex functions of one real variable, when the isotonicity
of the differential is automatic, for several variables, while this is not
necessarily true in the case of a differentiable convex function of a vector
variable. See \cite{N2021}, Remark 4.

In this paper we deal with functions defined on ordered Banach spaces, that
is, on real Banach spaces endowed with order relations $\leq$ that make them
ordered vector spaces such that the positive cones are closed and%
\[
\mathbf{0}\leq\mathbf{x}\leq\mathbf{y}\text{ implies }\left\Vert
\mathbf{x}\right\Vert \leq\left\Vert \mathbf{y}\right\Vert .
\]

The Euclidean $N$-dimensional space $\mathbb{R}^{N}$ has a natural structure
of ordered Banach space associated to the coordinatewise ordering. The usual
sequence spaces $c_{0},c,\ell^{p}$ (for $p\in\lbrack1,\infty])$ and the
function spaces $C(K)$ (for $K$ a compact Hausdorff space) and $L^{p}\left(
\mu\right)  $ (for $1\leq p\leq\infty$ and $\mu$ a $\sigma$-additive positive
measure$)$ are also examples of ordered Banach spaces (with respect to the
coordinatewise/pointwise ordering and the natural norms).

A map $T:E\rightarrow F$ between two ordered vector spaces is called
\emph{isotone} (or \emph{order preserving}) if%
\[
\mathbf{x}\leq\mathbf{y}\text{ in }E\text{ implies }T(\mathbf{x})\leq
T(\mathbf{y})\text{ in }F
\]
and \emph{antitone }(or\emph{ order reversing}) if $-T$ is isotone. When $T$
is a linear operator, $T$ is isotone if and only if $T$ maps positive elements
into positive elements (abbreviated, $T\geq0).$

For basic information on ordered Banach spaces see \cite{NO}. The interested
reader may also consult the classical books of Aliprantis and Tourky
\cite{AT2007} and Meyer-Nieberg \cite{MN}.

As was noticed by Amann \cite{Amann1974}, Proposition 3.2, p. 184, the
G\^{a}teaux differentiability offers a convenient way to recognize the
property of isotonicity of functions acting on ordered Banach spaces: the
positivity of the differential. We state here his result (following the
version given in \cite{N2021}, Lemma 4):

\begin{lemma}
\label{lemAmann_conv}Suppose that $E$ and $F$ are two ordered Banach space,
$C$ is a convex subset of $E$ with nonempty interior $\operatorname{int}C$ and
$\Phi:C\rightarrow F$ is a \emph{ }convex function, continuous on $C$ and
G\^{a}teaux differentiable on $\operatorname{int}C.$ Then $\Phi$ is isotone on
$C$ if and only if $\Phi^{\prime}(\mathbf{a})\geq0$ for all $\mathbf{a}%
\in\operatorname{int}C.$
\end{lemma}

\begin{remark}
\label{rem2}If the ordered Banach space $E$ has finite dimension, then the
statement of Lemma \emph{\ref{lemAmann_conv}} remains valid by replacing the
interior of $C$ by the relative interior of $C$. See \emph{\cite{NP2018}},
Exercise $6$, p. $81$.
\end{remark}

As was noticed in \cite{MST}, Example $7$, the function%
\[
\gamma:(-\infty,1]\rightarrow\mathbb{R},\quad\gamma(x)=-2x^{3}+5x^{2}+6x.
\]
is convex on $(-\infty,5/6]$, concave on $[5/6,1]$, and $m$-star-convex on
$(-\infty,1],$ with $m=27/28.$ The last assertion follows from a formula due
to Mocanu,
\[
m=\inf\Big\{\frac{x\gamma^{\prime}(x)-\gamma(x)}{y\gamma^{\prime}%
(x)-\gamma(y)}:y\gamma^{\prime}(x)-\gamma(y),\,x,y\in I\Big\},
\]
mentioned at the bottom of page 72 in \cite{MST}.

Proceeding like in Lemma \ref{lempersp}, one can prove that the function
associated to $\gamma,$
\[
\Upsilon:(-\infty,1]\times\lbrack1,\infty)\rightarrow\mathbb{R},\text{\quad
}\Upsilon(x,y)=-\frac{2x^{3}}{y^{2}}+\frac{5x^{2}}{y}+6x,
\]
is $27/28$-star-convex. The function $\Upsilon$ is also Gateaux
differentiable, with%
\[
(x,y)=\allowbreak\left(  \frac{1}{y^{2}}\left(  -6x^{2}+10xy+6y^{2}\right)
,\allowbreak\frac{x^{2}}{y^{3}}\left(  4x-5y\right)  \right)  .
\]

According to Lemma \ref{lemAmann_conv}, the map%
\[
d\Upsilon:(-\infty,1]\times\lbrack1,\infty)\subset\mathbb{R}^{2}%
\rightarrow\mathbb{R}^{2}%
\]
is isotone on the domain where $d^{2}\Upsilon=d(d\Upsilon)$ is positive, that
is, where the Hessian of $\Upsilon,$%
\[
\left(
\begin{array}
[c]{cc}%
-\frac{2}{y^{2}}\left(  6x-5y\right)  & 2\frac{x}{y^{3}}\left(  6x-5y\right)
\\
2\frac{x}{y^{3}}\left(  6x-5y\right)  & \allowbreak-2\frac{x^{2}}{y^{4}%
}\left(  6x-5y\right)
\end{array}
\right)  ,
\]
has all entries nonnegative. Therefore $d\Upsilon$ is isotone on
$(-\infty,1]\times\lbrack1,\infty).$

\section{The majorization relation on ordered Banach spaces}

In this section we discuss the concept of majorization into the framework of
ordered Banach spaces. Since in an ordered Banach space not every string of
elements admits a decreasing rearrangement, in this paper we will concentrate
to the case of pairs of discrete probability measures of which at least one of
them is supported by a monotone string of points. The case where the support
of the left measure consists of a decreasing string is defined as follows.

\begin{definition}
\label{def2}Suppose that $\sum_{k=1}^{N}\lambda_{k}\delta_{\mathbf{x}_{k}}$
and $\sum_{k=1}^{N}\lambda_{k}\delta_{\mathbf{y}_{k}}$ are two discrete Borel
probability measures that act on the ordered Banach space $E$ and $m\in(0,1]$
is a parameter. We say that $\sum_{k=1}^{N}\lambda_{k}\delta_{\mathbf{x}_{k}}$
is weakly $mL^{\downarrow}$-majorized by $\sum_{k=1}^{N}\lambda_{k}%
\delta_{\mathbf{y}_{k}}$ \emph{(}denoted $\sum_{k=1}^{N}\lambda_{k}%
\delta_{\mathbf{x}_{k}}\prec_{wmL^{\downarrow}}\sum_{k=1}^{N}\lambda_{k}%
\delta_{\mathbf{y}_{k}}$\emph{)} if the left hand measure is supported by a
decreasing string of points
\begin{equation}
\mathbf{x}_{1}\geq\cdots\geq\mathbf{x}_{N} \label{x}%
\end{equation}
and%
\begin{equation}
\sum_{k=1}^{n}\lambda_{k}\mathbf{x}_{k}\leq\sum_{k=1}^{n}\lambda
_{k}m\mathbf{y}_{k}\quad\text{for all }n\in\{1,\dots,N\}. \label{maj1}%
\end{equation}

We say that $\sum_{k=1}^{N}\lambda_{k}\delta_{\mathbf{x}_{k}}$ is
$mL^{\downarrow}$-majorized by $\sum_{k=1}^{N}\lambda_{k}\delta_{\mathbf{y}%
_{k}}$ $($denoted \linebreak$\sum_{k=1}^{N}\lambda_{k}\delta_{\mathbf{x}_{k}%
}\prec_{mL^{\downarrow}}\sum_{k=1}^{N}\lambda_{k}\delta_{\mathbf{y}_{k}})$ if
in addition%
\begin{equation}
\sum_{k=1}^{N}\lambda_{k}\mathbf{x}_{k}=\sum_{k=1}^{N}\lambda_{k}%
m\mathbf{y}_{k}. \label{maj2}%
\end{equation}

\end{definition}

Notice that the context of Definition \ref{def2} makes necessary that all
weights $\lambda_{1},...,\lambda_{N}$ belong to $(0,1]$ and $\sum_{k=1}%
^{N}\lambda_{k}=1.$

The three conditions (\ref{x}), (\ref{maj1}) and (\ref{maj2}) imply
$m\mathbf{y}_{N}\leq\mathbf{x}_{N}\leq\mathbf{x}_{1}\leq$ $m\mathbf{y}_{1}$
but not the ordering $\mathbf{y}_{1}\geq\cdots\geq\mathbf{y}_{N}.$ For
example, when $N=3,$ one may consider the case where
\[
m=1,\text{\ }\lambda_{1}=\lambda_{2}=\lambda_{3}=1/3,\text{ }\mathbf{x}%
_{1}=\mathbf{x}_{2}=\mathbf{x}_{3}=\mathbf{x}%
\]
and%
\[
\mathbf{y}_{1}=\mathbf{x},\text{ }\mathbf{y}_{2}=\mathbf{x+z,}\text{
}\mathbf{y}_{3}=\mathbf{x-z,}%
\]
$\mathbf{z}$ being any positive element.

Under these circumstances it is natural to introduce the following companion
to Definition \ref{def2}, involving the ascending strings of elements as
support for the right hand measure.

\begin{definition}
\label{def3}The relation of\emph{ weak }$mR^{\uparrow}$-\emph{majorization},
\[
\sum_{k=1}^{N}\lambda_{k}\delta_{\mathbf{x}_{k}}\prec_{wmR^{\uparrow}}%
\sum_{k=1}^{N}\lambda_{k}\delta_{\mathbf{y}_{k}},
\]
between two discrete Borel probability measures means the fulfillment of the
condition $($\emph{\ref{maj1}}$)$ under the presence of the ordering%
\begin{equation}
\mathbf{y}_{1}\leq\cdots\leq\mathbf{y}_{N}; \label{y}%
\end{equation}
assuming in addition the condition $($\emph{\ref{maj2}}$)$\emph{,} we say that
$\sum_{k=1}^{N}\lambda_{k}\delta_{\mathbf{x}_{k}}$ is $mR^{\uparrow}%
$-majorized by $\sum_{k=1}^{N}\lambda_{k}\delta_{\mathbf{y}_{k}}$ $($denoted
$\sum_{k=1}^{N}\lambda_{k}\delta_{\mathbf{x}_{k}}\prec_{mR^{\uparrow}}%
\sum_{k=1}^{N}\lambda_{k}\delta_{\mathbf{y}_{k}})$.
\end{definition}

When every element of $E$ is the difference of two positive elements, the weak
majorization relations $\prec_{mL^{\downarrow}}$and $\prec_{mR^{\uparrow}}$
can be augmented so to obtain majorization relations.

\section{The extension of the Hardy-Littlewood-Polya inequality of
majorization}

The objective of this section is to consider the corresponding extensions of
the Hardy-Littlewood-P\'{o}lya inequality of majorization for $\prec
_{wmL^{\downarrow}},\prec_{mL^{\downarrow}},\prec_{wmR^{\uparrow}}$and
$\prec_{mR^{\uparrow}}$. Moreover, we present also a Sherman's type inequality.

The proof of the following theorem is inspired by the techniques succesfully
used in \cite{MPP1995} and \cite{N2021}.

\begin{theorem}
\label{thmHLPgen}Suppose that $\sum_{k=1}^{N}\lambda_{k}\delta_{\mathbf{x}%
_{k}}$ and $\sum_{k=1}^{N}\lambda_{k}\delta_{\mathbf{y}_{k}}$ are two discrete
probability measures whose supports are included in an open convex subset $C$
of the ordered Banach space $E$. If $\sum_{k=1}^{N}\lambda_{k}\delta
_{\mathbf{x}_{k}}\prec_{mL^{\downarrow}}\sum_{k=1}^{N}\lambda_{k}%
\delta_{\mathbf{y}_{k}},$ then%
\begin{equation}
m\sum_{k=1}^{N}\lambda_{k}\Phi(\mathbf{y}_{k})\geq\sum_{k=1}^{N}\lambda
_{k}\Phi(\mathbf{x}_{k})+\sum_{k=1}^{N}\lambda_{k}\omega(\left\Vert
\mathbf{x}_{k}-\mathbf{y}_{k}\right\Vert ), \label{Cons1}%
\end{equation}
for every G\^{a}teaux differentiable $\omega$-$m$-star-convex function
$\Phi:C\rightarrow F$ whose differential is isotone and verifies the
hypotheses of Lemma \ref{lem1}.

The conclusion $(\ref{Cons1})$ still works under the weaker hypothesis
$\sum_{k=1}^{N}\lambda_{k}\delta_{\mathbf{x}_{k}}\prec_{wmL^{\downarrow}}%
\sum_{k=1}^{N}\lambda_{k}\delta_{\mathbf{y}_{k}},$ provided that $\Phi$ is
also an isotone function.
\end{theorem}

\

\begin{proof}
According to the gradient inequality \eqref{grad_ineq}, we have
\begin{multline*}
m\sum_{k=1}^{N}\lambda_{k}\Phi(\mathbf{y}_{k})-\sum_{k=1}^{N}\lambda_{k}%
\Phi(\mathbf{x}_{k})=\sum_{k=1}^{N}\lambda_{k}\left(  m\Phi(\mathbf{y}%
_{k})-\Phi(\mathbf{x}_{k})\right) \\
\geq\sum_{k=1}^{N}\Phi^{\prime}(\mathbf{x}_{k})(\lambda_{k}m\mathbf{y}%
_{k}-\lambda_{k}\mathbf{x}_{k})+\sum_{k=1}^{N}\lambda_{k}\omega\left(
\left\Vert \mathbf{x}_{k}-\mathbf{y}_{k}\right\Vert \right)  ,
\end{multline*}
whence, by using Abel's trick of interchanging the order of summation
(\cite{NP2018}, Theorem 1.9.5, p. 57), one obtains%
\begin{multline*}
\sum_{k=1}^{N}\lambda_{k}m\Phi(\mathbf{y}_{k})-\sum_{k=1}^{N}\lambda_{k}%
\Phi(\mathbf{x}_{k})-\sum_{k=1}^{N}\lambda_{k}\omega\left(  \left\Vert
\mathbf{x}_{k}-\mathbf{y}_{k}\right\Vert \right) \\
\geq\Phi^{\prime}(\mathbf{x}_{1})(\lambda_{1}m\mathbf{y}_{1}-\lambda
_{1}\mathbf{x}_{1})+\sum_{m=2}^{N}\Phi^{\prime}(\mathbf{x}_{m})\Bigl[\sum
_{k=1}^{m}(\lambda_{k}\mathbf{y}_{k}-\lambda_{k}\mathbf{x}_{k})-\sum
_{k=1}^{m-1}(\lambda_{k}\mathbf{y}_{k}-\lambda_{k}\mathbf{x}_{k})\Bigr]\\
=\sum_{m=1}^{N-1}\Bigl[(\Phi^{\prime}(\mathbf{x}_{m})-\Phi^{\prime}%
(\mathbf{x}_{m+1}))\sum_{k=1}^{m}(\lambda_{k}m\mathbf{y}_{k}-\lambda
_{k}\mathbf{x}_{k})\Bigr]+\Phi^{\prime}(\mathbf{x}_{N})\left(  \sum_{k=1}%
^{N}(\lambda_{k}m\mathbf{y}_{k}-\lambda_{k}\mathbf{x}_{k})\right)  .
\end{multline*}
When $\sum_{k=1}^{N}\lambda_{k}\delta_{\mathbf{x}_{k}}\prec_{mL^{\downarrow}%
}\sum_{k=1}^{N}\lambda_{k}\delta_{\mathbf{y}_{k}},$ the last term vanishes and
the fact that $D\geq0$ is a consequence of the isotonicity of $\Phi^{\prime}.$
When $\sum_{k=1}^{N}\lambda_{k}\delta_{\mathbf{x}_{k}}\prec_{wmL^{\downarrow}%
}\sum_{k=1}^{N}\lambda_{k}\delta_{\mathbf{y}_{k}}$ and $\Phi$ is isotone, one
applies Lemma \ref{lemAmann_conv} $(a)$ to infer that%
\[
\Phi^{\prime}(\mathbf{x}_{N})\left(  \sum_{k=1}^{N}(\lambda_{k}m\mathbf{y}%
_{k}-\lambda_{k}\mathbf{x}_{k})\right)  \geq0.
\]

The other cases can be treated in a similar way.
\end{proof}

\begin{remark}
Even in the context of usual convex functions, the isotonicity of the
differential is not only sufficient but also \emph{necessary} for the validity
of Theorem \ref{thmHLPgen}. See \cite{N2021}, Remark 5.
\end{remark}

We leave to the reader as an exercise the problem of formulating the variant
of Theorem \ref{thmHLPgen} in the case of relations $\prec_{wmR^{\uparrow}}%
$and $\prec_{mR^{\uparrow}}.$

\section{Further results and open problems}

The aim of this section is to mention some open problems which might be of
interest for further research on .

Notice first that any perturbation of an $\omega$-$m$-star-convex function
$\Phi$ satisfying the hypotheses of Theorem \ref{thmHLPgen} by a bounded
function $\Pi$ verify an inequality of majorization very close to
(\ref{Cons1}). Precisely, if $\left\vert \Pi\right\vert \leq\delta$ and
$\sum_{k=1}^{N}\lambda_{k}\delta_{\mathbf{x}_{k}}\prec_{mL^{\downarrow}}%
\sum_{k=1}^{N}\lambda_{k}\delta_{\mathbf{y}_{k}},$ then $\Psi=\Phi+\Pi$ will
verify the relation%
\[
m\sum_{k=1}^{N}\lambda_{k}\Phi(\mathbf{y}_{k})\geq\sum_{k=1}^{N}\lambda
_{k}\Phi(\mathbf{x}_{k})+\sum_{k=1}^{N}\lambda_{k}\omega(\left\Vert
\mathbf{x}_{k}-\mathbf{y}_{k}\right\Vert )-(1+m)\delta.
\]

This call the attention to the following class of \emph{approximately}
$\omega$-$m$-star-convex functions:

\begin{definition}
A function $\Phi:C\rightarrow\mathbb{R}$ is said to be $\delta$-$\omega$%
-$m$-star-convex function if it verifies an estimate of the form%
\[
\Phi((1-\lambda)\mathbf{x}+\lambda m\mathbf{y})\leq(1-\lambda)\Phi
(\mathbf{x})+m\lambda\Phi(\mathbf{y})-m\lambda(1-\lambda)\omega\left(
\left\Vert \mathbf{x}-\mathbf{y}\right\Vert \right)  +\delta,\text{\quad}%
\]
for some $\delta\geq0$ and all $\mathbf{x},\mathbf{y}\in C$ and $\lambda
\in(0,1).$
\end{definition}

The above definition extends (for $\omega=0$ and $m=1)$ the concept of
$\delta$-convex function, first considered by Hyers and Ulam \cite{HU1952} in
a paper dedicated to the stability of convex functions. It is natural to rise
the problem wheather their result extends to the framework of $\delta$%
-$\omega$-$m$-star-convex functions:

\begin{problem}
Suppose that $C$ is a convex subset of $\mathbb{R}^{N}.$ Is that true that
every $\delta$-$\omega$-$m$-star-convex function $\Phi:C\rightarrow\mathbb{R}$
can be written as $\Phi=\Psi+\Pi$, where $\Psi$ is an $\omega$-$m$-star-convex
function and $\Pi$ is a bounded function whose supremum norm is not larger
than $k_{N}\delta$ , where the positive constant $k_{N}$ depends only on the
dimension $N$ of the underlying space?
\end{problem}

Of some interest seems to be the concept of local approximate $m$%
-star-convexity suggested by \cite{DG2004}, Definition 1, which clearly yields
new extensions of the majorization inequality:

\begin{definition}
A function $\Phi:C\rightarrow\mathbb{R}$ is called locally approximately
$m$-star-convex if for every $x_{0}\in C$, and every $\varepsilon>0$ there
exists $\delta>0$ such that for all $x,y$ in the open ball of center $x_{0}$
and radius $\delta$ and all $\lambda\in(0,1),$%
\[
\Phi((1-\lambda)x+m\lambda y)\leq(1-\lambda)\Phi(x)+m\lambda\Phi(y)+\U{3b5}
t(1-t)\left\Vert x-y\right\Vert .
\]

\end{definition}

The whole discussion above can be placed in the more general context of
$M_{p}$-convexity.

Recall that the weighted $M_{p}$-mean is defined for every pair of positive
numbers $a,b$ by the formula%
\[
M_{p}(a,b;1-\lambda,\lambda)=\left\{
\begin{array}
[c]{cl}%
((1-\lambda)a^{p}+\lambda b^{p})^{1/p}, & \text{if }p\in\mathbb{R}%
\backslash\{0\}\\
a^{1-\lambda}b^{\lambda}, & \text{if }p=0\\
\max\{a,b\}, & \text{if }p=\infty,
\end{array}
\right.
\]
where $\lambda\in\lbrack0,1].$ If $p>0,$ then it is usual to extend $M_{p}$ to
all pairs of nonnegative numbers.

\begin{definition}
A function $\Phi:C\rightarrow\mathbb{R}$ is called $\omega$-$m$-$M_{p}%
$-star-convex if there exist a number $p\in\mathbb{R}$ and a modulus
$\omega:\mathbb{[}0,\infty\mathbb{)}\rightarrow\mathbb{R}$ such that
\begin{equation}
\Phi\left(  \left(  1-\lambda\right)  \mathbf{x}+\lambda\mathbf{y}\right)
\leq((1-\lambda)\Phi(\mathbf{x})^{p}+m\lambda\Phi(\mathbf{y})^{p}%
)^{1/p}-m\lambda(1-\lambda)\omega\left(  \left\Vert \mathbf{x}-\mathbf{y}%
\right\Vert \right)  ,\text{\quad}\nonumber\label{whm}%
\end{equation}
for all $\mathbf{x},\mathbf{y}\in C$ and $\lambda\in(0,1).$

Reversing the inequality one obtain the concept of $\omega$-$m$-$M_{p}%
$-star-concave functions.
\end{definition}

The usual $M_{p}$-convex/$M_{p}$-concave functions represent the particular
case where $m=1$ and $\omega=0.$

It is worth noticing that the $M_{p}$-convex ($M_{p}$-concave) functions for
$p\neq0$ are precisely the functions $\Phi$ such that $\Phi^{p}$ is convex
(concave), while the $M_{0}$-convex ($M_{0}$-concave) functions are nothing
but the $\log$-convex ($\log$-concave) functions. Notice also that the
$M_{\infty}$-convex ($M_{-\infty}$-concave) functions are precisely the
quasi-convex (quasi-concave) functions.

The next result represents the extension of Lemma \ref{lem1} to the case of
$\omega$-$m$-$M_{p}$-star-convex functions.

\begin{lemma}
\label{lem4}Suppose that $C$ is an open convex subset of the Banach space $E$
and $\Phi:C\rightarrow\mathbb{R}_{+}$ is a function both G\^{a}teaux
differentiable and $\omega$-$m$-$M_{p}$-star-convex. If $p\neq0,$ then $\Phi$
verifies the inequality \emph{ }
\[
\Phi^{p}(\mathbf{y})\geq\Phi^{p}(\mathbf{x})+p\Phi(\mathbf{x})^{p-1}%
d\Phi(\mathbf{x})(\mathbf{y}-\mathbf{x})+m\omega\left(  \left\Vert
\mathbf{x}-\mathbf{y}\right\Vert \right)  ,
\]
for all $\mathbf{x,y}\in C.$ 
\end{lemma}

The analogue of this result for $p=0$ and $\omega=0$ requires the strict
positivity of the function $\Phi$ and can be stated as%
\[
\log\Phi(\mathbf{y})-\log\Phi(\mathbf{x})\geq\frac{d\Phi(\mathbf{x}%
)(\mathbf{y}-\mathbf{x})}{\Phi(\mathbf{x})},
\]
for all $\mathbf{x,y}\in C.$ The last two inequalities work in the reverse
direction in the case of $\omega$-$m$-$M_{p}$-star-concave functions.

\ While it is clear that Lemma \ref{lem4} allows us to prove
Hardy-Littlewood-P\'{o}lya type inequalities more general than those provided
by Theorem \ref{thmHLPgen}, the exploration of the world of $\omega$%
-$m$-$M_{p}$-star-convex/concave functions for $\omega\neq0$ and $m\in(0,1)$
is just at the beginning..  

\

\textbf{Data availability statement:} The authors declare that  data supporting the findings of this study are available within the article and its supplementary information files.

\end{document}